\documentclass[reqno, 11pt]{amsart}
\usepackage[utf8]{inputenc}
\usepackage[T1]{fontenc}
\usepackage[english]{babel}
\usepackage{fullpage}
\usepackage{amssymb,amsmath,amscd,mathrsfs,graphicx,amsthm,stmaryrd,amsfonts,amsxtra,thmtools}
\usepackage{bm}
\usepackage{rotating,enumerate,float}
\usepackage{tikz-cd,tikz,verbatim,mathtools,multirow}
\usepackage[pdftex, dvipsnames]{xcolor}
\usepackage{hyperref}
\usepackage{cleveref}

\usepackage[sorting=nyc,doi=false,isbn=false,url=false,giveninits=true,maxbibnames=99]{biblatex}

\tikzstyle{nodal}=[circle,draw,fill=black,inner sep=0pt, minimum width=4pt]

 \DeclareFontFamily{U}{wncy}{}
     \DeclareFontShape{U}{wncy}{m}{n}{<->wncyr10}{}
     \DeclareSymbolFont{mcy}{U}{wncy}{m}{n}
     \DeclareMathSymbol{\Sh}{\mathord}{mcy}{"58} 

 \DeclareSortingScheme{nyc}{
   \sort{
     \field{presort}
   }
   \sort[final]{
     \field{sortkey}
   }
   \sort{
     \field{sortname}
     \field{author}
     \field{editor}
     \field{translator}
     \field{sorttitle}
     \field{title}
   }
   \sort{
     \field{sortyear}
     \field{year}
   }
   \sort{\citeorder}
 }

\newtheorem{theorem}{Theorem}[section]
\newtheorem{proposition}[theorem]{Proposition}
\newtheorem{lemma}[theorem]{Lemma}
\newtheorem{corollary}[theorem]{Corollary}

\theoremstyle{definition}

 \newtheorem*{acknowledgments*}{Acknowledgments}

\newtheorem*{ack}{\textbf{Acknowledgments}}

\theoremstyle{remark}
\newtheorem{remark}[theorem]{Remark}

\newcommand{\Hom}{\mathrm{Hom}}

\newcommand{\Aut}{\mathrm{Aut}}

\newcommand{\QQ}{\mathbb{Q}}
\newcommand{\ZZ}{\mathbb{Z}}

\newcommand{\RR}{\mathbb{R}}

\newcommand{\HH}{\mathbb{H}}

\newcommand{\Pic}{\mathrm{Pic}}
\DeclareMathOperator{\rk}{rk}

\newcommand{\Or}{\mathrm{O}}

\newcommand{\D}{\mathcal{D}}

\title{On symmetries of hyperbolic lattices of large rank}

\author{Torben Grabbel}
\address{Torben Grabbel: ETH Zürich, Rämistrasse 101, 8092 Zürich, Switzerland}
\email{tgrabbel@ethz.ch}

\author{Gebhard Martin}
 \address{Gebhard Martin: Mathematisches Institut \\ Universität Bonn \\ Endenicher Allee 60 \\ 53115 Bonn \\ Germany}
 \email{gmartin@math.uni-bonn.de} 

\author{Giacomo Mezzedimi}
 \address{Giacomo Mezzedimi: Mathematisches Institut \\ Universität Bonn \\ Endenicher Allee 60 \\ 53115 Bonn \\ Germany}
 \email{mezzedim@math.uni-bonn.de}

\author{Maia Raitz von Frentz}
\address{Maia Raitz von Frentz: Mathematisches Institut \\ Universität Göttingen \\ Bunsenstraße 3-5 \\ 37073 Göttingen \\ Germany }
\email{m.raitzvonfrentz@stud.uni-goettingen.de}

\author{Paul Jakob Schmidt}
\address{Paul Jakob Schmidt: Mathematisches Institut \\ Universität Bonn \\ Endenicher Allee 60 \\ 53115 Bonn \\ Germany}
\email{s07pschm@uni-bonn.de}

\date{\today}

\begin{document}

\maketitle

\begin{abstract}
For an even, integral hyperbolic lattice $L$, the symmetry group of $L$ is the quotient of the group of isometries of $L$ by the Weyl subgroup of $(-2)$-reflections. Following Nikulin, the exceptional lattice of $L$ is defined as the sublattice generated by elements that have finite orbit under the symmetry group of $L$. We prove that every hyperbolic lattice of rank at least $46$ has trivial exceptional lattice. In particular, every such lattice admits a symmetry of maximal Salem degree.
\end{abstract}

\section{Introduction}

\subsection{Symmetries and exceptional lattice}

Let $L$ be an even, integral hyperbolic lattice. The quotient $\Aut(\D_L)=\Or^+(L)/W(L)$ of the group of positive isometries of $L$ by the Weyl subgroup of reflections along $(-2)$-roots is called the \emph{symmetry group} of $L$. Foundational work by Nikulin \cite{nikulin.rank.greater.than.five,nikulin.rank.three} and Vinberg \cite{vinberg.rank.four} in the 80s gives a complete classification of hyperbolic lattices with a finite symmetry group. This was recently generalized by Yu \cite{yu.zero.entropy} (in the case of Picard lattices of K3 surfaces) and Brandhorst--Mezzedimi \cite{brandhorst.mezzedimi.borcherds,brandhorst.mezzedimi.abelian}, who classified hyperbolic lattices with a \emph{virtually abelian} symmetry group, that is, containing an abelian subgroup of finite index. The maximum rank of such a hyperbolic lattice is $26$, attained by the unimodular lattice ${\rm II}_{1,25} \coloneqq U\oplus \Lambda$, where $\Lambda$ denotes the Leech lattice.

In order to capture the complexity of symmetry groups of general hyperbolic lattices, Nikulin introduced in \cite{nikulin.elliptic.fibrations} the notion of \emph{exceptional lattice}: If $L$ is a hyperbolic lattice, then the exceptional lattice $E(L)\subseteq L$ is the sublattice generated by elements that have finite orbit under $\Aut(\D_L)$. Following \cite[§4]{nikulin.elliptic.fibrations}, there are three possibilities for the exceptional lattice if it is non-trivial: 
\begin{enumerate}
    \item $E(L)$ is a hyperbolic lattice. Then $E(L)=L$, and in particular $\Aut(\D_L)$ is finite.
    \item $E(L)$ is negative semidefinite. Then it contains an isotropic vector, and therefore $L$ is a \emph{Borcherds lattice} (see \cite[Definition~1.1]{brandhorst.mezzedimi.borcherds}). In particular $\Aut(\D_L)$ is virtually abelian.
    \item $E(L)$ is negative definite. 
\end{enumerate}

As explained above, the lattices satisfying (1) or (2) have already been classified. Moreover, Nikulin proves in \cite[Theorem~4.4]{nikulin.elliptic.fibrations} that there are only finitely many hyperbolic lattices of \emph{fixed} rank $\ge 3$ with non-trivial exceptional lattice. 
The goal of this article is to make the next step towards a full classification also in case (3), by proving that $E(L)$ is in fact trivial in large enough rank. More precisely, we will show the following:

\begin{theorem}\label{thm:maintheorem}
    Let $L$ be a hyperbolic lattice of rank at least $46$. Then, the exceptional 
    lattice of $L$ is trivial. 
\end{theorem}

In particular, the classification of hyperbolic lattices with non-trivial exceptional lattice becomes a finite problem:

\begin{corollary} \label{cor:finitelymany}
    Up to isometry, there are only finitely many hyperbolic lattices of rank $\ge 3$ with non-trivial exceptional lattice.
\end{corollary}

Observe that in rank $2$, a hyperbolic lattice has non-trivial exceptional lattice if and only if it has finite symmetry group, or equivalently if and only if it represents $0$ or $-2$ (see \cite[Corollary~3.4]{galluzzi.lombardo.peters}).

Given the finiteness result in \Cref{cor:finitelymany}, it is natural to ask what is the largest rank of a hyperbolic lattice with non-trivial exceptional lattice. If $L\coloneqq {\rm II}_{1,25}$, then already Borcherds noted in \cite{borcherds.leech} that $E(L)$ has rank $1$, and it is generated by the unique cusp (i.e. primitive, isotropic vector in the closure of the fundamental domain $\D_L$) with infinite stabilizer. We do not know any example of larger rank. In the spirit of the classifications in \cite{nikulin.rank.greater.than.five,nikulin.rank.three,vinberg.rank.four,brandhorst.mezzedimi.borcherds}, it would be interesting to have a complete classification of hyperbolic lattices with a non-trivial exceptional lattice, thus completing Nikulin's program in \cite{nikulin.elliptic.fibrations}.

\subsection{Motivation from K3 surfaces and applications} 
If $L = \Pic(X)$ is the Picard lattice of a complex K3 surface $X$, then the classes of effective $(-2)$-curves in $L$ determine a choice of fundamental domain $\D_L$ for the action of $W(L)$ on $L$ and the associated action
\begin{eqnarray*}
\Aut(X) & \to & \Aut(\D_L) \\
g & \mapsto & g^*
\end{eqnarray*}
has finite kernel and cokernel. The \emph{cusps} of $L$ that we consider in this article correspond to elliptic fibrations of $X$ and cusps with infinite stabilizer correspond to elliptic fibrations whose Jacobian has infinite Mordell--Weil group. If $X$ admits at least one elliptic fibration with infinite automorphism group, then the exceptional lattice is the primitive sublattice of $\Pic(X)$ generated by all $(-2)$-curves that are contained in fibers of all such elliptic fibrations (cf. \Cref{prop:exceptionalIntersection}).

To measure the complexity of an automorphism $g$, one considers the characteristic polynomial of $f=g^*$, which factors over $L_{\mathbb{R}}$ as $S(x)C(x)$, where $S(x)$ is a Salem polynomial (or a quadratic polynomial or trivial), called the \emph{Salem factor}, and $C(x)$ is a product of cyclotomic polynomials (see \cite[Discussion before D\'efinition 1.2]{cantat.dynamique}). The Salem factor is particularly interesting, since the logarithm of its largest root coincides with the topological entropy of $X$ \cite{Gromov,Yomdin}. Thus, to measure the complexity of this number, we define the \emph{Salem degree} of $f$ as the degree of $f$ and, since this degree is always even, we say that $f$ has \emph{maximal Salem degree} if its Salem degree is $\rk(L)$ if $\rk(L)$ is even or $\rk(L)-1$ if $\rk(L)$ is odd.

Now, observe that these definitions do not use the fact that there is an underlying K3 surface $X$ at all. In fact, even the choice of $\D_L$, which is needed to consider $\Aut(\D_L)$ as a subgroup of $\Or^+(L)$, is irrelevant, since any two choices lead to conjugate subgroups.
Our results on hyperbolic lattices can therefore be used to apply statements in the literature about automorphisms of K3 surfaces in terms of the exceptional lattice to the symmetry group $\Aut(\D_L)$. One such statement is the following, due to Yu \cite[Theorem 6.2]{yu.elliptic.fibrations}. The original proof in the case of K3 surfaces of even Picard rank extends easily to the abstract setting and we can rephrase the result in our terminology as follows:

\begin{theorem}[Yu]
Let $L$ be a hyperbolic lattice with trivial exceptional lattice $E(L)$. Then:
\begin{enumerate}
    \item The only $\Aut(\D_L)$-invariant $\RR$-linear subspaces of $L_\RR$ are $\{0\}$ and $L_\RR$.
    \item There exists a symmetry $f\in \Aut(\D_L)$ of maximal Salem degree.
\end{enumerate}
If moreover $\rk(L)\ge 4$ is even, then the conditions $E(L)=0$, (1) and (2) are all equivalent.
\end{theorem}

By \Cref{rmk: cusp exists}, there always exists a cusp of infinite stabilizer if ${\rm rk}(L) \geq 20$ and by our main \Cref{thm:maintheorem}, the exceptional lattice is trivial as soon as ${\rm rk}(L) \geq 46$. We thus obtain the following application to Salem degrees:

\begin{corollary}
Let $L$ be a hyperbolic lattice of rank at least $46$. Then, the following hold:
\begin{enumerate}
\item There are no non-trivial $\Aut(\D_L)$-invariant $\RR$-linear subspaces of $L_{\mathbb{R}}$.
\item There exists an isometry $f \in \Aut(\D_L)$ of maximal Salem degree.
\end{enumerate}
\end{corollary}

\subsection{Strategy of the proof and outline}

The starting point of the proof of \Cref{thm:maintheorem} is an explicit description of the exceptional lattice of hyperbolic lattices, originally due to Nikulin \cite[Theorem~4.2]{nikulin.elliptic.fibrations} (see also \Cref{prop:exceptionalIntersection}). More precisely, it asserts that $E(L)$ is trivial as soon as every simple $(-2)$-root $r\in L$ admits a cusp with infinite stabilizer $e\in \overline{\D}_L$ such that $e.r>0$.

Now let $L$ be a hyperbolic lattice of rank at least $46$. By taking an arbitrary maximal overlattice of $L$, we reduce to lattices of the form $L=U\oplus E_8^n\oplus M$, where $n\ge 5$ and the rank of $M$ is at most $11$ (cf. \Cref{prop:structure}). By using the isometry $U\oplus E_8^3\cong U\oplus \Lambda$, and the fact that the Leech lattice $\Lambda$ has no $(-2)$-root, we deduce that the exceptional lattice $E(L)$ is contained in $M$. Since there are only finitely many possibilities for the root part of $M$, we are able to show in \Cref{section:proof} that $E(L)=0$ for all of them, thus concluding the proof of \Cref{thm:maintheorem}.

Given the nature of our approach, we expect the bound in \Cref{thm:maintheorem} to be quite far from sharp. Moreover, it is probable that a (possibly computer assisted) refinement of our argument in \Cref{section:proof} via extended Dynkin diagrams leads to a much better bound in \Cref{thm:maintheorem}.

We briefly outline the structure of the paper. In Section 2 we recall basic properties of negative definite and hyperbolic lattices, and prove the structure result for maximal overlattices of hyperbolic lattices in \Cref{prop:structure}. In Section 3 we review Nikulin's definition of the exceptional lattice, and the explicit description in \Cref{prop:exceptionalIntersection}. Finally, in Section 4 we prove \Cref{thm:maintheorem}.

\begin{ack}
We gratefully acknowledge the hospitality and support of the Max-Planck-Institut f\"ur Mathematik Bonn, where we started working on this article during an internship program.
\end{ack}

\section{Preliminaries on lattices}

\subsection{Notation}

A \emph{lattice} $L$ is a free $\ZZ$-module of finite rank with a symmetric, non-degenerate bilinear form $(-.-): L \times L \to \mathbb{Z}$. For a lattice $L$, denote by $L(m)$ for 
$m\in \ZZ$ the scaled lattice, that is, $L(m)$ has the same underlying group as $L$ but 
$(u.v)_{L(m)}=m\cdot (u.v)_L$ for every pair of vectors $u,v$. We call $L$ \emph{even} 
if $(u.u)$ is even for all $u \in L$, and odd otherwise. To shorten notation, we will assume 
that all lattices are even.

The \emph{rank} $\mathrm{rk}(L)$ of a lattice $L$ is the rank of $L$ as a $\ZZ$-module. 
The \emph{signature} of a lattice is the signature of the corresponding form on 
$L\otimes\RR$. A lattice is called \emph{hyperbolic} if its signature is 
$(1,\mathrm{rk}(L)-1)$.

The \emph{discriminant} $\mathrm{disc}(L)$ of a lattice $L$ is the absolute value of 
the determinant of the Gram matrix with respect to any basis of $L$. A lattice $L$ is 
called \emph{unimodular} if $\mathrm{disc}(L)=1$. The \emph{dual lattice} is defined as 
$L^{\lor}=\Hom( L,\ZZ)\subset L\otimes\QQ$,
with the natural restriction of the bilinear form on $L\otimes \QQ$. The 
\emph{discriminant group} $A_L$ is defined as the finite abelian group $L^{\lor}/L$ 
with its finite quadratic form $q$ with values in $\QQ/2\ZZ$. The discriminant 
$\mathrm{disc}(L)$ is equal to $|A_L|$. The \emph{length} $\ell(A_L)$ is defined as the 
minimum number of generators of $A_L$, while the \emph{$p$-length} $\ell_p(A_L)$ is 
defined as the minimum number of generators of the $p$-Sylow subgroup of $A_L$. Observe 
that the length $\ell(A_L)$ coincides with the maximum of the $p$-lengths. 

A \textit{sublattice} of a lattice $L$ is a lattice $M$ with an embedding 
$M\hookrightarrow L$ such that the bilinear form on $L$ restricts to that on $M$. If additionally $\rk M=\rk L$ (or equivalently $[L:M]<\infty$), we call $L$ an \textit{overlattice} of $M$.
It is a standard fact (found e.g. in \cite[Prop.~1.4.1]{Ni80}) that overlattices of 
$L$ correspond to subgroups of $A_L$ on which the quadratic form is trivial. The \emph{saturation} of a sublattice $M\subseteq L$ is defined as $M_{\mathrm{pr}}=(M\otimes \mathbb{Q})\cap L$, i.e. the largest overlattice of $M$ that fits into $M\hookrightarrow M_{\mathrm{pr}}\hookrightarrow L$.

Given two lattices $L_1$ and $L_2$, their \textit{direct sum} $L_1\oplus L_2$ is simply their direct sum as abelian groups, with the bilinear form obtained by setting $(u_1.u_2)=0$ whenever $u_1\in L_1$, $u_2\in L_2$ and such that the inclusions $L_1,L_2 \hookrightarrow L_1 \oplus L_2$ are sublattices.
Note that if a unimodular lattice $M$ is a sublattice of a lattice $L$, then it is already a direct summand. Indeed, for $u\in L$ we can consider the functional $\psi_u\colon M \to \ZZ$ such that $\psi_u(v)=(u.v)$. By unimodularity of $M$ there exists a unique $u_M\in M$ such that $(u_M.v)=\psi_u(v)=(u.v)$ for every $v\in M$, so $u-u_M\in M^\perp$ and thus the homomorphism $L\to M \oplus M^\perp$ sending $u$ to $(u_M,u-u_M)$ is an isomorphism.

Denote by $U$ the unique hyperbolic, unimodular lattice of rank $2$. Two lattices with 
the same signature $L,L'$ are said to be \emph{in the same genus} if $A_L$ and $A_{L'}$ are 
isomorphic as finite quadratic modules. This is equivalent to the condition that the 
lattices $U\oplus L$ and $U\oplus L'$ are isometric (see \cite[Corollary~1.13.4]{Ni80}).

\subsection{Roots and root lattices}

A \emph{root} of a lattice $L$ is an element $r\in L$ of norm $-2$. The \emph{root part} $L_{\mathrm{root}}$ of a lattice $L$ is the sublattice generated by its roots. 
A negative definite lattice is a \emph{root lattice} if it 
is equal to its root part.
In this section, let $R$ be a root lattice.
We denote the finite set of roots of $R$ by $\Delta_R$. 

Every root lattice decomposes as the direct sum of ADE lattices, consisting of the 
infinite series $A_n$, $n\ge 1$, $D_n$, $n\ge 4$, and the three lattices $E_6, E_7$ and 
$E_8$ \cite[Theorem~1.2]{Eb13}. We will always consider ADE lattices with the standard $\ZZ$-basis $\{r_1,\ldots,r_n\}$ of roots given by the corresponding Dynkin diagram (see \cite[Theorem~1.1]{Eb13}). If a root lattice $R$ decomposes as a direct sum $R_1\oplus \ldots \oplus R_m$, 
where the $R_i$ are ADE lattices, we observe that 
$\Delta_R = \Delta_{R_1} \cup \ldots \cup \Delta_{R_m}$.

We will need to understand for which root lattices $R$, the sum $U \oplus R$ contains $U \oplus E_8$. If $R \neq E_8$, we obviously need ${\rm rk}(R) \geq 9$. This is enough if $R$ is an ADE lattice:

\begin{proposition} \label{prop: big roots contain UE8}
Let $R$ be an ADE lattice of rank $n \geq 9$. Then, $U \oplus R$ contains $U \oplus E_8$.
\end{proposition}
\begin{proof}
From the classification of ADE lattices, we have $R \in \{A_n,D_n\}$, hence $\ell(A_R) \leq 2$ and even $\ell(A_R) = 1$ unless $R = D_n$ with $n$ even (see \cite[Table~1.1]{Eb13}).
Using \cite[Corollary 1.13.5]{Ni80} and the inequality
$$
{\rm rk}(U \oplus R) \geq 11 \geq 9 + {\ell}(A_R),
$$
we conclude that $U \oplus R \cong E_8 \oplus M$ for a hyperbolic lattice $M$ of rank $n -6 \geq 3$ with $\ell(A_M) = \ell(A_R)$. By \cite[Corollary 1.13.5]{Ni80}, we are done if $n - 6 \geq 3 + \ell(A_R)$. 
It thus remains to consider the cases $R \in \{D_9,D_{10},A_9\}$. A straightforward computation shows that the discriminant group of $D_9$ (resp. $D_{10}$, $A_9$) is isomorphic as a finite quadratic module to the discriminant group of $\langle -4\rangle$ (resp. $A_1\oplus A_1$, $\langle -10\rangle$), see e.g. \cite[Table~1]{gvirtz.mezzedimi}. Therefore $D_9$ (resp. $D_{10}$, $A_9$) is in the same genus of $E_8\oplus \langle -4\rangle$ (resp. $E_8\oplus A_1\oplus A_1$, $E_8\oplus \langle -10\rangle$), from which
$$U\oplus D_9\cong U\oplus E_8\oplus \langle -4\rangle, \quad U\oplus D_{10} \cong U\oplus E_8\oplus A_1\oplus A_1, \quad U\oplus A_9\cong U\oplus E_8\oplus \langle -10\rangle. $$ 
\end{proof}

Every ADE lattice $R$ admits a \emph{highest root}, which can be defined as the unique root $r_0$ of $R$ such that $r_0.r_i \le 0$ for every root in the standard basis $\{r_1,\ldots,r_n\}$.
For the convenience of the reader, we recall the coefficients of the highest root of each ADE lattice in \Cref{highest.roots}, and refer to \cite[§1.5]{Eb13} for more details.

  \begin{figure}[h!]
    \centering
    \begin{tikzpicture}[scale=0.8, label distance=2]
        \begin{scope}[xshift=-3cm]
        \node (A0) at (-1,0) [nodal, label=below:$1$] {};
        \node (A1) at (-2,0) [nodal, label=below:$1$] {};
        \node (B0) at ( 1,0) [nodal, label=below:$1$] {};
        \node (B1) at ( 2,0) [nodal, label=below:$1$] {};
        \draw (A0)--(A1) (B0)--(B1);
        \draw [dashed] (A0)--(B0);
        \end{scope}
        \begin{scope}[xshift=5cm]
        \node (A0) at (-1,0) [nodal, label=below:$2$] {};
        \node (A1) at (-2,0) [nodal, label=below:$1$] {};
        \node (B0) at ( 1,0) [nodal, label=below:$2$] {};
        \node (B1) at ( 2,0) [nodal, label=below:$2$] {};
        \node (C1) at (3,1) [nodal, label=right:$1$] {};
        \node (C2) at (3,-1) [nodal, label=right:$1$] {};
        \draw (A0)--(A1) (B0)--(B1) (C1)--(B1)--(C2);
        \draw [dashed] (A0)--(B0);
        \end{scope}
        \begin{scope}[yshift=-3cm]
        \begin{scope}[xshift=-6cm]
        \node (C0) at ( 0,0) [nodal, label=below:$3$] {};
        \node (A1) at (-1,0) [nodal, label=below:$2$] {};
        \node (A2) at (-2,0) [nodal, label=below:$1$] {};
        \node (B1) at ( 1,0) [nodal, label=below:$2$] {};
        \node (B2) at ( 2,0) [nodal, label=below:$1$] {};
        \node (C1) at ( 0,1) [nodal, label=right:$2$] {};
        \draw (C0)--(A1)--(A2) (C0)--(B1)--(B2) (C0)--(C1);
        \end{scope}
        \begin{scope}
        \node (C0) at ( 0,0) [nodal, label=below:$4$] {};
        \node (A1) at (-1,0) [nodal, label=below:$3$] {};
        \node (A2) at (-2,0) [nodal, label=below:$2$] {};
        \node (B1) at ( 1,0) [nodal, label=below:$3$] {};
        \node (B2) at ( 2,0) [nodal, label=below:$2$] {};
        \node (B3) at ( 3,0) [nodal, label=below:$1$] {};
        \node (C1) at ( 0,1) [nodal, label=right:$2$] {};
        \draw (C0)--(A1)--(A2) (C0)--(B1)--(B2)--(B3) (C0)--(C1);
        \end{scope}
        \begin{scope}[xshift=7cm]
        \node (C0) at ( 0,0) [nodal, label=below:$6$] {};
        \node (A1) at (-1,0) [nodal, label=below:$4$] {};
        \node (A2) at (-2,0) [nodal, label=below:$2$] {};
        \node (B1) at ( 1,0) [nodal, label=below:$5$] {};
        \node (B2) at ( 2,0) [nodal, label=below:$4$] {};
        \node (B3) at ( 3,0) [nodal, label=below:$3$] {};
        \node (B4) at ( 4,0) [nodal, label=below:$2$] {};
        \node (C1) at ( 0,1) [nodal, label=right:$3$] {};
        \draw (C0)--(A1)--(A2) (C0)--(B1)--(B2)--(B3)--(B4) (C0)--(C1);
        \end{scope}
        \end{scope}
    \end{tikzpicture}
    \caption{Highest roots of ADE lattices}
    \label{highest.roots}
    \end{figure}
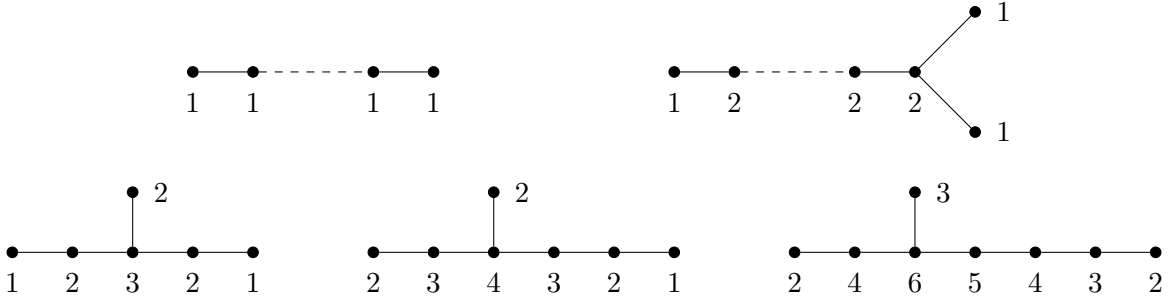

To each ADE lattice $R$, we associate the \emph{extended ADE lattice} $\widetilde R$. For the precise construction, we refer to \cite[§1.5,~pp.~29--30]{Eb13}. Let us only recall that $\widetilde R$ is negative semidefinite, of rank $\rk(R) + 1$, obtained by adjoining a unique root $\tilde r$ to $R$. Moreover, if $r_0\in R$ denotes the highest root, then the only isotropic vectors of $\widetilde R$ are multiples of $e \coloneqq \tilde r + r_0$, and $e$ generates the kernel of the bilinear form associated to $\widetilde R$. 

\begin{lemma} \label{lemma:roots.of.extended.ADE}
\vbox{Let $R$ be an ADE lattice with standard basis $\{r_1,\ldots,r_n\}$, $\widetilde R$ the corresponding extended ADE lattice obtained by adjoining the root $\tilde r$ to $R$. Every root $r\in \widetilde R$ can be written as
$$r=\sum_{i=1}^n{\alpha_i r_i} + \beta \tilde r$$
for integers $\alpha_i,\beta$, either all non-negative or all non-positive.}
\end{lemma}
\begin{proof}
Since $e$ is in the kernel of the quadratic form of $\widetilde R$, we have $r=\alpha e + r'$, with $r'\in R$ a root. By \cite[Lemma~1.3]{Eb13}, $r'$ is a linear combination of $r_1,\ldots,r_n$ with all non-negative (or non-positive) coefficients. Up to changing the sign of $r$, we may assume that all the coefficients of $r'$ are non-negative. If $\alpha\ge 0$, then
$$r = \alpha e + r' = \alpha \tilde{r} + \alpha r_0 + r',$$
where $r_0\in R$ is the highest root, and the claim follows. If instead $\alpha < 0$, then 
$$r = \alpha e + r' = \alpha \tilde{r} + \alpha r_0 + r' = \alpha \tilde{r} + (\alpha +1)r_0 + (r'-r_0),$$
and the coefficients of $r'-r_0$ are all non-positive by \cite[Lemma~1.13]{Eb13}.
\end{proof}

\subsection{Roots, Weyl group and symmetry group} \label{subsection:symmetry}

In this section we will assume that $L$ is a hyperbolic lattice. The positive cone 
$\mathcal{P}_L$ is a connected component of the set $\{x\in L\otimes\RR : (x.x)>0\}$ and 
$\HH_L:=\{x\in\mathcal{P}_L : (x.x)=1\}$ is the corresponding hyperboloid model of hyperbolic 
space of dimension $\mathrm{rk}(L)-1$. Passing to the conformal ball model of $\HH_L$, 
there is a natural identification between elements of the boundary $\partial\HH_L$ and 
isotropic rays in $L\otimes \RR$. Let $\Or^+(L)$ denote the group of isometries of $L$ 
preserving the connected component $\mathcal{P}_L$. 

For any root $r\in L$ we define the \emph{reflection} associated to $r$ as the isometry of $L$ given by
\[\begin{array}{cccc}
\phi_r: &L &\rightarrow &L\\
&v &\mapsto &v+(v.r)r.
\end{array}\]
The \emph{Weyl group} $W(L)$ of $L$ is defined as the subgroup of $\Or^+(L)$ generated by 
reflections associated to roots. On the hyperbolic space $\HH_L$ these reflections can 
be seen as reflections along the subspaces $r^{\perp}\cap\HH_L$, which cut $\HH_L$ into 
connected components called \emph{chambers}. The Weyl group acts simply transitively on 
these chambers. Observe that $W(L)$ is a normal subgroup of $\Or^+(L)$. Moreover, the short 
exact sequence
\[
    0 \to W(L) \to \Or^+(L) \to \Or^+(L)/W(L) \to 0
\]
splits, and subgroups $\mathrm{Aut}(\mathcal{D}_L)\subset \Or^+(L)$ fixing a chamber $\mathcal{D}_L$ correspond to sections of the projection $\Or^+(L)\twoheadrightarrow \Or^+(L)/W(L)$.
Since the Weyl group $W(L)$ acts simply transitively on the set of 
chambers, the subgroups $\mathrm{Aut}(\mathcal{D}_L)$ of $\Or^+(L)$, for different choices 
of chamber $\mathcal{D}_L$, are all conjugate.

We fix a chamber $\mathcal{D}_L$ of $\HH_L$, which by construction is a fundamental domain for the action of $W(L)$ on $\HH_L$. The group $\Aut(\mathcal D_{L}) $ is called 
the \emph{symmetry group} of $L$, and isometries in $\Aut(\mathcal D_{L}) $ are called 
\emph{symmetries} of $L$.
A vector $v\in  L\otimes\RR$ is said to be \emph{fundamental} if 
$v\in\RR^+\cdot\overline{\mathcal{D}}_L$. 
A vector $v\neq0$ is said to be \emph{positive} if $v$ has non-negative inner product with 
all fundamental vectors.

\begin{lemma} \label{lemma:positive.root}
If $0\ne v\in L$ is a vector of norm $v^2\ge -2$, then exactly one of $v$ and $-v$ is positive. More precisely, $v$ is positive if and only if $v$ has positive inner product with \emph{one} fundamental vector.
\end{lemma}
\begin{proof}
Since $v^2\ge -2$, the hyperplane $v^\perp$ is disjoint from $\D_L$ (if $v^2\ge 0$, $v^\perp$ is even disjoint from $\HH_L$), and therefore $\D_L$ lies on a unique side of the hyperplane $v^\perp$. Hence $v$ is positive if and only if $v$ has positive inner product with one fundamental vector.
\end{proof}

A positive root is said to be \emph{simple} if it is \emph{indecomposable}, i.e it cannot 
be written as the sum of two positive vectors. Observe that, if $r, r'\in L$ are simple roots 
such that $r.r' < 0$, then $r= r'$. Indeed, if $r.r'<0$, then $(r-r')^2 \ge -2$, and therefore either $r-r'$ or $r'-r$ is positive, a contradiction unless $r=r'$.

A \emph{cusp} of $L$ is a primitive, fundamental, isotropic vector $e\in L$. Observe that the orthogonal complement $e^\perp$ is negative semidefinite, with a $1$-dimensional kernel generated by $e$, and the lattice $e^\perp/\langle e \rangle$ is negative definite. Moreover the projection $e^\perp\to e^\perp/\langle e \rangle$ preserves the bilinear forms.

\begin{lemma} \label{lemma:distinct.cusps}
Let $e,f \in L \otimes \mathbb{R}$ be linearly independent fundamental vectors. Then, $e.f >0$. In particular, if $e,f$ are distinct cusps, then $e.f>0$.
\end{lemma}
\begin{proof}
Assume first $e^2>0$. Since the hyperplane $e^\perp$ is disjoint from $\overline{\D}_L$, the vectors $e$ and $f$ lie on the same side of $e^\perp$. Since $e.e>0$, it follows that $e.f>0$.

Assume instead $e^2=0$. By the above paragraph, $e$ intersects every vector in $\D_L$ positively, so $e$ is positive by \Cref{lemma:positive.root}. Since $f$ is fundamental, it follows that $e.f\ge 0$. If $e.f=0$, then $f\in e^\perp$, and its image in $e^\perp/\langle e \rangle$ has non-negative self-intersection. However $e^\perp/\langle e \rangle$ is negative definite, so $f$ is a multiple of $e$, a contradiction.
\end{proof}

The following lemma is essential when working with fundamental domains:

\begin{lemma}\label{lem:summandFundamental}
	Let $L$ be a hyperbolic lattice, $M$ a negative definite lattice, and view $\HH_L$ as the subspace of $\HH_{L\oplus M}$ of vectors with zero projection onto $M$.
    There is a choice of fundamental domains $\D_L$ for $L$ and $\D_{L\oplus M}$ for $L\oplus M$ such that $\overline{\D}_L = \overline{\D}_{L\oplus M}\cap \HH_L$.
\end{lemma}
\begin{proof}
Choose a fundamental domain $\D_L$ for $L$ and any $v_L\in \D_L$. Denote $v=(v_L,0)\in L\oplus M$, and let $\D_{L\oplus M}$ be a fundamental domain for $L\oplus M$ such that $v\in \overline{\D}_{L\oplus M}$. The set of positive roots of $L$ is a subset of the set of positive roots of $L\oplus M$: indeed, if $r\in L$ is a positive root, then $v.r = v_L.r>0$ by assumption, so $r$ is positive in $L\oplus M$ by \Cref{lemma:positive.root}. It follows that $\overline{\D}_{L\oplus M}\cap \HH_L \subseteq \overline{\D}_L$. For the other direction, let $w\in\overline{\D}_L$. In order to show that $w\in \overline{\D}_{L\oplus M}$, it suffices to show that $w.r\ge 0$ for every positive root $r\in L\oplus M$. Write $r=z_1+z_2$, with $z_1\in L$ and $z_2\in M$. If $z_1=0$, then $w.r=0$, so without loss of generality assume $z_1\ne 0$. Moreover $M$ is negative definite, so $z_1^2\ge -2$, and thus either $z_1$ or $-z_1$ is positive in $L$. Since $v.r = v.z_1 \ne 0$ by assumption, by positivity of $r$ we deduce $v.r = v.z_1 > 0$, i.e. $z_1$ is positive in $L$. Hence $w. r = w.z_1 \ge 0$, as desired.
\end{proof}

\subsection{Isometries and Salem polynomials}

Let $L$ be a hyperbolic lattice, with fundamental domain $\D_L$. As explained in \Cref{subsection:symmetry}, the symmetry group $\Aut(\D_L)$ can be identified with the group of isometries of $L$ preserving the fundamental domain $\D_L$. Therefore $\Aut(\D_L)$ is a subgroup of the group of isometries $\Or(\HH_L)$.

Every isometry $f$ of the hyperbolic space $\HH_L$ has a characteristic polynomial $P(x)\in \ZZ[x]$ with at most one root outside the unit disk (cf. \cite[Discussion before D\'efinition 1.2]{cantat.dynamique}). In particular we can write $P(x)=S(x)C(x)$, where $S(x)$ is either trivial, a quadratic polynomial, or a \emph{Salem polynomial}, and $C(x)$ is a product of cyclotomic polynomials. We call $S(x)$ the \emph{Salem factor} of $f$, and its degree is the \emph{Salem degree} of $f$. Observe that the Salem degree of $f$ is zero if and only if all roots of $P$ are roots of unity, in which case we say that $f$ has \emph{zero entropy}. Recall that Salem polynomials always have even degree. Hence, we say that $f$ has \emph{maximal Salem degree} if its Salem degree is $\rk(L)$ (if $\rk(L)$ is even) or $\rk(L)-1$ (if $\rk(L)$ is odd).

\subsection{Structure theory of hyperbolic lattices}

We conclude the section with a structure result for maximal overlattices of hyperbolic lattices. The following lemma is well-known to experts. For lack of a precise reference, we give a proof.

\begin{lemma}\label{lem:lengthbound}
    If $L$ is a lattice of length $\ell(A_L)\ge 4$, then it admits a non-trivial overlattice.
\end{lemma}
\begin{proof}
By assumption there is a prime $p$ such that the $p$-length $\ell_p(A_L)$ is at least $4$.
Since finding a non-trivial overlattice amounts to finding a non-zero isotropic vector in $A_L$, we may restrict to the $p$-torsion subgroup of $A_L$ and assume that $A_L$ is a $p$-elementary group.
Every element $\bar{v}\in A_L$ can be written as $\bar{v}=\frac{v}{p}$, where $v\in L$ is a vector of \emph{divisibility} multiple of 
$p$ (i.e. $p$ divides $u.v$ for every $u\in L$). In particular, if $\bar{v},\bar{w}\in A_L$, then $\bar{v}.\bar{w}\in \frac{1}{p}\ZZ$.

Assume $p>2$. Up to multiplying by $p$, the equation $\bar{v}^2 = 0\in \QQ/2\ZZ$ can be rewritten as $Q(v)\equiv 0 \pmod{p}$, where $Q$ is a non-degenerate quadratic form over $\ZZ$. Since the $p$-length of $A_L$ is at least $3$, the quadratic form $Q$ has at least $3$ variables, and therefore the equation $Q(v)\equiv 0 \pmod{p}$ admits a non-trivial solution by the Chevalley--Warning theorem.

Assume now $p=2$. As above, the equation $\bar{v}^2 = 0\in \QQ/2\ZZ$ can be rewritten as 
$Q(v)\equiv 0 \pmod{4}$, where $Q$ is a non-degenerate quadratic form over $\ZZ$ in at least $4$ variables. By the Chevalley--Warning theorem, there 
exists $v_1$ such that $Q(v_1) \equiv 0 \pmod{2}$. If $4\mid Q(v_1)$, this yields a 
non-trivial isotropic vector in $A_L$. If instead $Q(v_1) \equiv 2 \pmod{4}$, the orthogonal complement $v_1^\perp$ has dimension $\ge 3$, and therefore there exists a vector $v_2$, orthogonal to $v_1$, such that $Q(v_2) \equiv 0 \pmod{2}$. Either 
$4\mid Q(v_2)$ or $4\mid Q(v_1+v_2)$, giving in any case a non-trivial isotropic vector in $A_L$. 
\end{proof}

\begin{remark}
The above bound is optimal. Indeed, the root lattice $A_1^{\oplus 3}$ admits no non-trivial even overlattice. Nevertheless, the same proof shows that $L$ admits a non-trivial overlattice as soon as $\ell_p(A_L)\ge 3$ for an odd prime $p$.
\end{remark}

\begin{proposition}\label{prop:structure}
	If $L$ is a hyperbolic lattice of rank at least $6$, then there exist an overlattice 
    $L'$ of $L$, an integer $n$ and a negative-definite lattice $M$ such that
    $$L'\cong U\oplus E_{8}^{n}\oplus M,$$
    with $\rk M\le 11$ and such that the root part of $M$ does not contain any $E_{8}$ or an ADE lattice of rank at least $9$.
\end{proposition}

\begin{proof}
Let $L'$ be a maximal overlattice of $L$. Then, $\ell(A_{L'}) \leq 3$ by \Cref{lem:lengthbound}.
As
$\rk(L')\ge 6$, it follows from \cite[Corollary~1.13.5]{Ni80} that $L'$ contains a copy of $U$. Let $n$ be maximal integer $\ge 0$ such that $U\oplus E_8^n$ embeds into $L'$.

Since $U\oplus E_8^n$ is unimodular, we obtain a decomposition $L'=U\oplus E_8^n\oplus M$ for some $M$, and 
    $\ell(A_M)=\ell(A_{L'})\le 3$. By 
    \cite[Corollary~1.13.5]{Ni80}, if $\rk M\ge 12$, then $U\oplus M$ has a direct summand 
    $E_8$ and its complement has rank at least $6$ (and length at most $3$), so by 
    the same Corollary, $U\oplus M\cong U\oplus E_8\oplus M'$ for some $M'$, a contradiction to the maximality of $n$. Hence $\rk M\le 11$. Since $n$ is chosen maximal, the statement about the root part follows immediately from \Cref{prop: big roots contain UE8}.
\end{proof}

\section{The exceptional lattice}

In this section we review the notion of exceptional lattice, introduced by Nikulin in the context of K3 surfaces and, more in general, of arbitrary hyperbolic lattices.

Let $L$ be a hyperbolic lattice, with symmetry group $\Aut(\D_L)$.
Following Nikulin \cite[§4]{nikulin.elliptic.fibrations}, we define the \emph{exceptional lattice} 
$E(L)$ of $L$ to be the sublattice of vectors $v\in L$ with finite orbit under 
$\mathrm{Aut}(\mathcal{D}_L)$. More generally, if $G\subseteq \Aut(\mathcal{D}_L)$ is a subgroup, we denote $E_G(L)$ the sublattice of vectors $v\in L$ with finite orbit under $G$. 

We observe that the exceptional lattice behaves well with respect to overlattices:

\begin{lemma}\label{lem:overlatticeChambers}
Let $L$ be a hyperbolic lattice, and $L'$ an overlattice of $L$. There exist chambers $\mathcal{D}_L$ of $L$ and $\mathcal{D}_{L'}$ of $L'$ such that $\mathcal{D}_{L'} \subseteq \mathcal{D}_L$. In particular, $E(L)\subseteq E(L')$.
\end{lemma}
\begin{proof}
The first statement is clear, since $L$ and $L'$ have the same rank and any root of $L$ is also a root of $L'$. 
If $G\subseteq \Aut(\mathcal{D}_{L'})$ is the finite index subgroup of isometries preserving $L\subseteq L'$, by the first part of the statement we have an injective homomorphism $G\hookrightarrow \Aut(\mathcal{D}_L)$. 
Since $G\subseteq \Aut(\mathcal{D}_{L'})$ is of finite index, we have $E_G(L')=E(L')$, and clearly $E(L)\subseteq E_G(L) \subseteq E_G(L')$, concluding the proof.
\end{proof}

Let $\mathcal E(L)$ denote the set of cusps in $L$ and let $\mathcal E_{\infty}(L)$ denote the subset of cusps with infinite stabilizer in $\mathrm{Aut}(\mathcal{D}_L)$. For a cusp $e\in \mathcal{E}(L)$, we denote by $(e^\perp)^{(2)}$ the sublattice of $L$ generated by $e$ and the roots orthogonal to $e$, and by $(e^{\perp})^{(2)}_{\mathrm{pr}}$ the saturation of $(e^\perp)^{(2)}$ in $L$.

\begin{remark} \label{rmk:take.care.of.degenerate.lattice}
Since the lattice $e^\perp$ is degenerate with kernel $e$, it decomposes as in \cite[Lemma~4.1]{brandhorst.mezzedimi.borcherds} into the direct sum $\langle e \rangle \oplus W$, where the lattice $W\coloneqq e^\perp/\langle e \rangle$ is negative definite. Then, we can view $W$ as a primitive sublattice of $L$ by composing $W \hookrightarrow e^{\perp} \hookrightarrow L$.

If $R\subseteq W$ is any primitive sublattice, then $\langle e \rangle \oplus R$ is primitive in $e^\perp$, and therefore in $L$. Indeed a lattice embedding is primitive if and only if the cokernel is torsion-free, and
$$e^\perp/(\langle e \rangle \oplus R) = (\langle e \rangle \oplus W)/(\langle e \rangle \oplus R) = W/R.$$
Therefore, if $R\coloneqq W_{\mathrm{root},\rm pr}$ is the saturation (in $W$) of the lattice spanned by the roots of $W$, it holds that $(e^{\perp})^{(2)}_{\mathrm{pr}} = \langle e \rangle \oplus R$.
\end{remark}

Recall that, by a generalization of the Shioda--Tate formula (see e.g. \cite[Corollary~1.5.4]{nikulin.rank.greater.than.five}, and \cite[Proposition~3.2]{brandhorst.mezzedimi.borcherds} for a complete proof), a cusp $e\in \mathcal{E}(L)$ has infinite stabilizer if and only if $\rk (e^\perp)^{(2)}<\rk e^\perp$, or equivalently if and only if $e^\perp/e$ is not an overlattice of a root lattice.

The following explicit description of the exceptional lattice was stated by Nikulin in \cite[Theorem~4.2]{nikulin.elliptic.fibrations}, though he only proved it for Picard lattices of K3 surfaces (see \cite[Theorem~4.1]{nikulin.elliptic.fibrations}). For the sake of completeness, we include a proof.

\begin{proposition}\label{prop:exceptionalIntersection}
Let $L$ be a lattice such that $\mathcal E_\infty(L)$ is non-empty. 
Then the exceptional lattice satisfies
\begin{align*}
    E(L)=\bigcap _{e\in \mathcal E_{\infty}(L)} (e^{\perp})^{(2)}_{\mathrm{pr}}.
\end{align*}
\end{proposition}
\begin{proof}
In order to prove the inclusion $\supseteq$, it suffices to show that all the vectors in
$$E'\coloneqq \bigcap _{e\in \mathcal E_{\infty}(L)} (e^{\perp})^{(2)}_{\mathrm{pr}}$$
have finite orbit under $\Aut(\D_L)$. Clearly $\Aut(\D_L)$ sends $E'$ to $E'$. If $\mathcal{E}_\infty(L)$ consists of a unique cusp $e$, then $\Aut(\D_L)$ fixes it; moreover $E'=(e^{\perp})^{(2)}_{\mathrm{pr}} = \langle e \rangle \oplus R$, where $R$ is a negative definite lattice by \Cref{rmk:take.care.of.degenerate.lattice}, and therefore $\Aut(\D_L)$ acts on $R$ (and thus on $E'$) with finite order.
If instead $\mathcal{E}_\infty(L)$ contains at least two distinct cusps, say $e_1$, $e_2$, then $(e_1+e_2)^2 = 2(e_1.e_2) > 0$ and therefore $E'\subseteq \langle e_1 + e_2\rangle^\perp$ is negative definite, so again $\Aut(\D_L)$ acts on $E'$ with finite order.

For the other inclusion, it suffices to show that $E(L)\subseteq (e^\perp)^{(2)}_{\rm pr}$ 
for every cusp $e$ with infinite stabilizer. In fact, we will prove the stronger inclusion
$$E_{\Aut(\mathcal{D}_L,e)}(L) \subseteq (e^\perp)^{(2)}_{\rm pr},$$
where $\Aut(\mathcal{D}_L,e)$ denotes the stabilizer of $e$ in $\Aut(\mathcal{D}_L)$.

Let $v\in E_{\Aut(\mathcal{D}_L,e)}(L)$. Take any $g\in \Aut(\mathcal{D}_L,e)$ of infinite 
order. Up to taking a power of $g$, by assumption we may assume that $g(v)=v$, that is, $g$ 
acts as the identity on the rank two sublattice $\langle e,v \rangle$ spanned by $e$ and 
$v$. If $(e.v) \ne 0$, then the lattice $\langle e,v \rangle$ is hyperbolic, and therefore 
$g$ induces an isometry of infinite order of the negative definite lattice 
$\langle e,v \rangle^\perp$, a contradiction. Hence 
$E_{\Aut(\mathcal{D}_L,e)}(L)\subseteq e^\perp$.

If $W=e^\perp/\langle e \rangle$, $R=W_{\mathrm{root},\rm pr}\subseteq W$, and $R^\perp_W$ denotes the orthogonal complement of $R$ in $W$, then by \Cref{rmk:take.care.of.degenerate.lattice} we have
$(e^\perp)^{(2)}_{\rm pr} = \langle e \rangle \oplus R$ and the inclusion
$$\langle e \rangle \oplus R \oplus R^\perp_W = (e^\perp)^{(2)}_{\rm pr} \oplus R^\perp_W \subseteq  e^\perp$$
is of finite index. Since $(e^\perp)^{(2)}_{\rm pr} \subseteq E_{\Aut(\mathcal{D}_L,e)}(L)$ by the first part of the proof, in order to conclude it suffices to show 
that every $0\ne y\in R^\perp_W$ has infinite orbit under $\Aut(\mathcal{D}_L,e)$. For this, 
consider the Eichler--Siegel transformation
$$\psi_y: x \mapsto x + (x.y)e - (x.e)y-\frac{1}{2}(x.e)(y.y) e.$$
It follows from \cite[Proposition~3.2]{brandhorst.mezzedimi.borcherds} that $\psi_y$ is a symmetry in $\Aut(\mathcal{D}_L,e)$. Since $\psi_y(e) = e$ and $\psi_y(y) = y + (y.y)e$, we obtain
$$\psi_y^m(y) = y + m (y.y) e \ne y$$
for every $m> 0$ by induction, so $y$ has infinite orbit under $\Aut(\mathcal{D}_L,e)$, thus completing the proof.
\end{proof}

\begin{remark} \label{rmk: cusp exists}
\begin{enumerate}
    \item  Every hyperbolic lattice of rank $\ge 20$ admits a cusp with infinite stabilizer. Indeed, every hyperbolic lattice of rank $\ge 20$ has an infinite symmetry group (cf. \cite{nikulin.rank.greater.than.five}), and therefore admits a cusp with infinite stabilizer by \cite[Theorem~6.4.1]{nikulin.rank.greater.than.five}.

    \item Note that \Cref{prop:exceptionalIntersection} fails if $L$ admits no cusps with infinite stabilizer. As an example, let $L$ a hyperbolic lattice of rank $3$ with $\Aut(\D_L)\cong \ZZ$ and no cusp with infinite stabilizer. Such lattices exist, and were classified in \cite{brandhorst.mezzedimi.abelian}; for instance we can take $L=U(11)\oplus A_1$. Then $\Aut(\D_L)$ is generated by a symmetry $f$ with infinite order and Salem degree $2$, and thus $f$ preserves a line in $L$. Hence $E(L)$ has rank $1$, while $E'$ is $L$ itself.
\end{enumerate}
\end{remark}

\section{Proof of \Cref{thm:maintheorem}} \label{section:proof}

We now prove \Cref{thm:maintheorem}. Throughout the section, $L$ is a hyperbolic lattice, and we fix a fundamental domain $\D_L$ for $L$.
We start with the following well-known observation.

\begin{lemma}\label{lem:highestCusp}
Let $L$ be a hyperbolic lattice, and $\Delta$ a finite set of simple roots of $L$, whose associated intersection graph is an extended Dynkin diagram. Denote by $e_\Delta $ the unique positive primitive isotropic vector in the extended ADE lattice $\ZZ[\Delta] \subseteq L$. Then $e_\Delta$ is a cusp in $L$.
\end{lemma}
\begin{proof}
It suffices to show that $(e_\Delta.r) \ge 0$ for every simple root $r$ of $L$. If $r$ is one of the roots in $\Delta$, then $(e_\Delta.r)=0$. If instead $r\notin \Delta$, then $r$ has nonnegative inner product with all the simple roots in $\Delta$, and in particular $(e_\Delta.r) \ge 0$.
\end{proof}

The following is a key observation for our approach:

\begin{lemma}\label{lem:easyCusps}
If $L=U\oplus E_8^n\oplus M$ with $n\ge 5$ and $M$ negative definite, then $E(L)\subseteq M$.
\end{lemma}
\begin{proof}
Recall that there exists an isometry $U\oplus E_8^3\cong U\oplus \Lambda$, where $\Lambda$ is the Leech lattice.
 In particular, there exists a cusp $e\in U\oplus E_8^3$ such that $e^\perp/\langle e \rangle \cong \Lambda$ has no roots, and therefore $(e^\perp)^{(2)}_{\rm pr} = \langle e \rangle$.

For each triple $(i,j,k)$, with $i<j<k$ and $i,j,k\in \{1,\ldots,n\}$, consider a cusp $e_{ijk} \in U\oplus E_8^3$ (where the three copies of $E_8$ are $i$-th, $j$-th and $k$-th direct summand of type $E_8$ of $L$) as in the previous paragraph such that $e_{ijk}^\perp/\langle e_{ijk} \rangle \cong \Lambda$, and view it as a cusp in $L$ by \Cref{lem:summandFundamental}. By the Shioda--Tate formula \cite[Corollary~1.5.4]{nikulin.rank.greater.than.five}, $e_{ijk}$ has infinite stabilizer in $L$.
Moreover
$$(e_{123}^\perp)^{(2)}_{\rm pr} \subseteq \langle e_{123} \rangle \oplus E_8^{n-3} \oplus M,$$
where the copies of $E_8$ are the last $n-3$. Since $E(L)$ is generated (over $\QQ$) by roots by \Cref{prop:exceptionalIntersection}, it suffices to show that any positive root $r\in E(L)$ belongs to $M$. Since $e_{123}$ has infinite stabilizer, $r\in (e_{123}^\perp)^{(2)}_{\rm pr} \subseteq \langle e_{123} \rangle \oplus E_8^{n-3} \oplus M$ by \Cref{prop:exceptionalIntersection}. Since $E_8$ and $M$ are negative definite, the root $r$ belongs to either $\langle e_{123} \rangle \oplus E_8$ (for one of the $n-3$ copies of $E_8$) or $\langle e_{123} \rangle \oplus M$; in other words $r$ is one of the roots in the extended Dynkin diagram $\widetilde{E}_8$ of $E_8$ or of an ADE lattice in $M_{\mathrm{root}}$.

Assume first that $r\in \langle e_{123} \rangle\oplus E_8$ for the $i$-th copy of $E_8$. If $\{r_1,\ldots,r_8\}$ is the standard basis of $E_8$, with $r_0\in E_8$ the highest root, then we consider the basis $\{r_1,\ldots,r_8,\tilde r\}$ of roots of $\langle e_{123} \rangle\oplus E_8$, with $\tilde{r}=e_{123}-r_0$. By \Cref{lemma:roots.of.extended.ADE} we may assume that $r\in \{r_1,\ldots,r_8,\tilde r\}$. Observe that, since the Leech lattice $\Lambda$ has no root, we have $r_j.e_{12i} > 0$ for every $1\le j\le 8$, so $r_j\notin E(L)$. If instead $r = \tilde{r}$, consider the cusp $e_{12j}$ for some $j\ge 3$, $j\ne i$. It is orthogonal to the $i$-th copy of $E_8$, but $e_{12j}.e_{12i}>0$ by \Cref{lemma:distinct.cusps}, so $e_{12j}.\tilde{r}>0$, and therefore $\tilde{r}\notin E(L)$. 

Assume now that $r\in \langle e_{123} \rangle \oplus M_0$ for an ADE factor $M_0$ of $M_{\mathrm{root}}$. If $r$ is the root $e_{123}-r_0$, for the highest root $r_0$ of $M_0$, then $e_{124}.r>0$, since $M\subseteq e_{124}^\perp$ and $e_{123}.e_{124}>0$ by \Cref{lemma:distinct.cusps}. In particular we have shown that
$$\bigcap_{i\ge 3}{(e_{12i}^{\perp})^{(2)}_{\mathrm{pr}}} \subseteq M,$$
and therefore $E(L)\subseteq M$ by \Cref{prop:exceptionalIntersection}.
\end{proof}

We are now ready to prove \Cref{thm:maintheorem}.

\begin{proof}[Proof of \Cref{thm:maintheorem}]
Let $L$ any hyperbolic lattice of rank $\ge 46$. By \Cref{prop:structure}, there exists an overlattice $L'$ of $L$ such that
$$L'\cong U\oplus E_8^n \oplus M,$$
with $n\ge 5$ and $M$ negative definite of rank at most $11$, and by \Cref{lem:overlatticeChambers} it suffices to prove the theorem for $L'$.

By \Cref{lem:easyCusps}, the exceptional lattice $E(L')$ is contained in $M$, and in particular it is negative definite.
Therefore by \Cref{prop:exceptionalIntersection} it suffices to show that, for any simple root $r$ in $M$, there exists a cusp $e\in \mathcal{E}_\infty(L')$ with infinite stabilizer such that $e.r>0$.

By using the isometry 
$$L'\cong U\oplus \Lambda \oplus E_8^{n-3}\oplus M,$$
where $\Lambda$ denotes the Leech lattice, we only need to find a cusp $e$ in
$$L''\coloneqq U\oplus E_8^{n-3}\oplus M = \Lambda^\perp \subseteq L'$$
such that $(e.r)>0$, since $e$ is a cusp in $L'$ with infinite stabilizer by \Cref{lem:overlatticeChambers}

We can make use of some final reduction steps: first, since $n-3 \ge 2$, we may assume that $n-3 = 2$. Secondly, the root part $M_{\mathrm{root}}$ of $M$ decomposes as a direct sum of ADE lattices $R_1,\ldots, R_k$, and the set of simple roots in $M$ is the union of the simple roots in each $R_i$, so without loss of generality we may assume that $M_{\mathrm{root}}$ is an ADE lattice. Finally, by \Cref{prop:structure} we may assume by maximality of $n$ that $M_{\mathrm{root}}$ is of type $A_m$ or $D_m$ with $m\le 8$, or of type $E_6$ or $E_7$. 

It remains to prove the following statement: if $L''=U\oplus E_8^2\oplus M$, where the root part $M_{\mathrm{root}}$ of $M$ is one of the ADE lattices above, and $r\in M$ is a simple root, then there exists a cusp $e$ of $L''$ such that $(e.r)>0$. We will use the following strategy: the only cusp $e_0\in U$ induces the dual graph $\Gamma$ of simple roots of $L''$ contained in $e_0^\perp$.
The graph $\Gamma$ contains two extended Dynkin diagrams of type $\widetilde{E}_8$, an extended Dynkin diagram of type $\widetilde{M}_{\mathrm{root}}$, and the simple root $s_0\in U$ intersecting simply a vertex of multiplicity $1$ of each extended Dynkin diagram. For each vertex $r$ of the Dynkin diagram of $M_{\mathrm{root}}$, we will exhibit an extended Dynkin diagram $\Delta\subseteq \Gamma$, whose associated primitive positive isotropic vector $e_\Delta$ (which is a cusp by \Cref{lem:highestCusp}) satisfies $(e_\Delta.r) >0$.

Consider first a vertex at distance $2$, $3$, $4$ or $6$ from $s_{0}$ in the Dynkin diagram of $M_{\mathrm{root}}$. In these cases, there is an extended Dynkin diagram of shape $\widetilde{D}_{10}$, $\widetilde E_{6}$, $\widetilde E_{7}$ and $\widetilde E_{8}$, respectively, intersecting the vertex positively, as can be seen in the following diagrams:

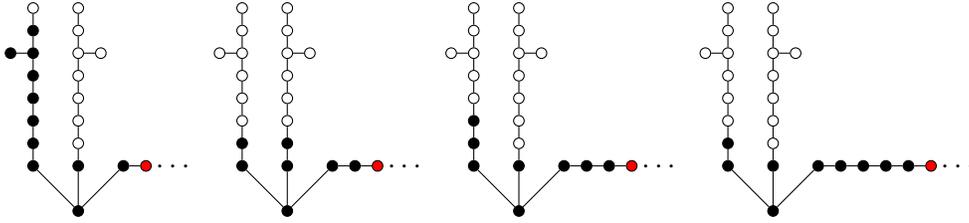
\begin{center}
\begin{tikzpicture}[scale=0.3]
\node (S) at (0,0) [nodal] {};

\node (T0) at (-2,2) [nodal] {};
\node (T1) at (-2,3) [nodal] {};
\node (T2) at (-2,4) [nodal] {};
\node (T3) at (-2,5) [nodal] {};
\node (T4) at (-2,6) [nodal] {};
\node (T5) at (-2,7) [nodal] {};
\node (T6) at (-3,7) [nodal] {};
\node (T7) at (-2,8) [nodal] {};
\node (T8) at (-2,9) [nodal,fill=white] {};
\draw(T0)--(T1)--(T2)--(T3)--(T4)--(T5)--(T6) (T5)--(T7)--(T8);
\draw (S)--(T0);

\node (U0) at (0,2) [nodal] {};
\node (U1) at (0,3) [nodal,fill=white] {};
\node (U2) at (0,4) [nodal,fill=white] {};
\node (U3) at (0,5) [nodal,fill=white] {};
\node (U4) at (0,6) [nodal,fill=white] {};
\node (U5) at (0,7) [nodal,fill=white] {};
\node (U6) at (1,7) [nodal,fill=white] {};
\node (U7) at (0,8) [nodal,fill=white] {};
\node (U8) at (0,9) [nodal,fill=white] {};
\draw(U0)--(U1)--(U2)--(U3)--(U4)--(U5)--(U6) (U5)--(U7)--(U8);
\draw (S)--(U0);

\node (S0) at (2,2) [nodal] {};
\node (S1) at (3,2) [nodal, fill=red] {};
\node (S2) at (4.3,1.95) {$\cdots$};
\draw(S0)--(S1);
\draw (S)--(S0);

\end{tikzpicture}
\begin{tikzpicture}[scale=0.3]
\node (S) at (0,0) [nodal] {};

\node (T0) at (-2,2) [nodal] {};
\node (T1) at (-2,3) [nodal] {};
\node (T2) at (-2,4) [nodal,fill=white] {};
\node (T3) at (-2,5) [nodal,fill=white] {};
\node (T4) at (-2,6) [nodal,fill=white] {};
\node (T5) at (-2,7) [nodal,fill=white] {};
\node (T6) at (-3,7) [nodal,fill=white] {};
\node (T7) at (-2,8) [nodal,fill=white] {};
\node (T8) at (-2,9) [nodal,fill=white] {};
\draw(T0)--(T1)--(T2)--(T3)--(T4)--(T5)--(T6) (T5)--(T7)--(T8);
\draw (S)--(T0);

\node (U0) at (0,2) [nodal] {};
\node (U1) at (0,3) [nodal] {};
\node (U2) at (0,4) [nodal,fill=white] {};
\node (U3) at (0,5) [nodal,fill=white] {};
\node (U4) at (0,6) [nodal,fill=white] {};
\node (U5) at (0,7) [nodal,fill=white] {};
\node (U6) at (1,7) [nodal,fill=white] {};
\node (U7) at (0,8) [nodal,fill=white] {};
\node (U8) at (0,9) [nodal,fill=white] {};
\draw(U0)--(U1)--(U2)--(U3)--(U4)--(U5)--(U6) (U5)--(U7)--(U8);
\draw (S)--(U0);

\node (S0) at (2,2) [nodal] {};
\node (S1) at (3,2) [nodal] {};
\node (S2) at (4,2) [nodal, fill=red] {};
\node (S3) at (5.3,1.95) {$\cdots$};
\draw(S0)--(S1)--(S2);
\draw (S)--(S0);

\end{tikzpicture}
\begin{tikzpicture}[scale=0.3]
\node (S) at (0,0) [nodal] {};

\node (T0) at (-2,2) [nodal] {};
\node (T1) at (-2,3) [nodal] {};
\node (T2) at (-2,4) [nodal] {};
\node (T3) at (-2,5) [nodal,fill=white] {};
\node (T4) at (-2,6) [nodal,fill=white] {};
\node (T5) at (-2,7) [nodal,fill=white] {};
\node (T6) at (-3,7) [nodal,fill=white] {};
\node (T7) at (-2,8) [nodal,fill=white] {};
\node (T8) at (-2,9) [nodal,fill=white] {};
\draw(T0)--(T1)--(T2)--(T3)--(T4)--(T5)--(T6) (T5)--(T7)--(T8);
\draw (S)--(T0);

\node (U0) at (0,2) [nodal] {};
\node (U1) at (0,3) [nodal,fill=white] {};
\node (U2) at (0,4) [nodal,fill=white] {};
\node (U3) at (0,5) [nodal,fill=white] {};
\node (U4) at (0,6) [nodal,fill=white] {};
\node (U5) at (0,7) [nodal,fill=white] {};
\node (U6) at (1,7) [nodal,fill=white] {};
\node (U7) at (0,8) [nodal,fill=white] {};
\node (U8) at (0,9) [nodal,fill=white] {};
\draw(U0)--(U1)--(U2)--(U3)--(U4)--(U5)--(U6) (U5)--(U7)--(U8);
\draw (S)--(U0);

\node (S0) at (2,2) [nodal] {};
\node (S1) at (3,2) [nodal] {};
\node (S2) at (4,2) [nodal] {};
\node (S3) at (5,2) [nodal, fill=red] {};
\node (S4) at (6.3,1.95) {$\cdots$};
\draw(S0)--(S1)--(S2)--(S3);
\draw (S)--(S0);

\end{tikzpicture}
\begin{tikzpicture}[scale=0.3]
\node (S) at (0,0) [nodal] {};

\node (T0) at (-2,2) [nodal] {};
\node (T1) at (-2,3) [nodal] {};
\node (T2) at (-2,4) [nodal,fill=white] {};
\node (T3) at (-2,5) [nodal,fill=white] {};
\node (T4) at (-2,6) [nodal,fill=white] {};
\node (T5) at (-2,7) [nodal,fill=white] {};
\node (T6) at (-3,7) [nodal,fill=white] {};
\node (T7) at (-2,8) [nodal,fill=white] {};
\node (T8) at (-2,9) [nodal,fill=white] {};
\draw(T0)--(T1)--(T2)--(T3)--(T4)--(T5)--(T6) (T5)--(T7)--(T8);
\draw (S)--(T0);

\node (U0) at (0,2) [nodal] {};
\node (U1) at (0,3) [nodal,fill=white] {};
\node (U2) at (0,4) [nodal,fill=white] {};
\node (U3) at (0,5) [nodal,fill=white] {};
\node (U4) at (0,6) [nodal,fill=white] {};
\node (U5) at (0,7) [nodal,fill=white] {};
\node (U6) at (1,7) [nodal,fill=white] {};
\node (U7) at (0,8) [nodal,fill=white] {};
\node (U8) at (0,9) [nodal,fill=white] {};
\draw(U0)--(U1)--(U2)--(U3)--(U4)--(U5)--(U6) (U5)--(U7)--(U8);
\draw (S)--(U0);

\node (S0) at (2,2) [nodal] {};
\node (S1) at (3,2) [nodal] {};
\node (S2) at (4,2) [nodal] {};
\node (S3) at (5,2) [nodal] {};
\node (S4) at (6,2) [nodal] {};
\node (S5) at (7,2) [nodal, fill=red] {};
\node (S6) at (8.3,1.95) {$\cdots$};
\draw(S0)--(S1)--(S2)--(S3)--(S4)--(S5);
\draw (S)--(S0);

\end{tikzpicture}

\end{center}

Note that this settles the cases $A_m$ with $m \leq 6$ and $D_m$ with $m \leq 5$. Next, given a diagram of type $\widetilde{A}_m$ or $\widetilde{D}_m$, with $m \geq 6$, we will exhibit extended Dynkin diagrams whose associated cusps intersect nodes at distance $5$ from $s_{0}$, namely of shape $\widetilde E_7$ and $\widetilde E_8$. This takes care of all simple roots of $A_m$, $m\le 8$ and all simple roots of $D_m$, $m\le 7$. For $D_8$ we additionally exhibit a cusp intersecting the furthest vertices, at distance $7$ from $s_0$, coming from an extended Dynkin diagram of shape $\widetilde E_8$.

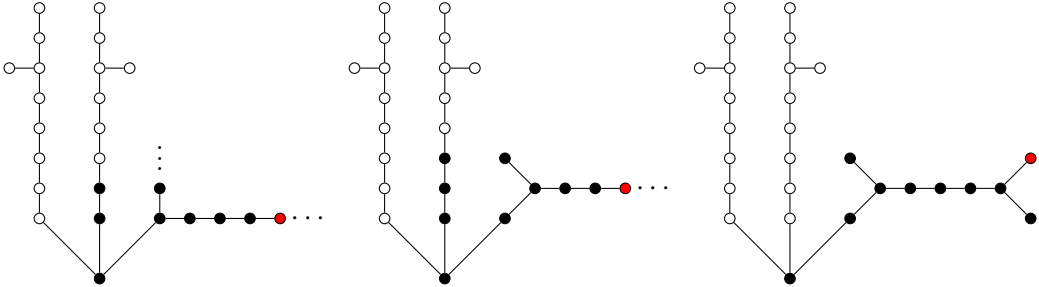
\begin{center}
\begin{tikzpicture}[scale=0.4]
\node (S) at (0,0) [nodal] {};

\node (T0) at (-2,2) [nodal,fill=white] {};
\node (T1) at (-2,3) [nodal,fill=white] {};
\node (T2) at (-2,4) [nodal,fill=white] {};
\node (T3) at (-2,5) [nodal,fill=white] {};
\node (T4) at (-2,6) [nodal,fill=white] {};
\node (T5) at (-2,7) [nodal,fill=white] {};
\node (T6) at (-3,7) [nodal,fill=white] {};
\node (T7) at (-2,8) [nodal,fill=white] {};
\node (T8) at (-2,9) [nodal,fill=white] {};
\draw(T0)--(T1)--(T2)--(T3)--(T4)--(T5)--(T6) (T5)--(T7)--(T8);
\draw (S)--(T0);

\node (U0) at (0,2) [nodal] {};
\node (U1) at (0,3) [nodal] {};
\node (U2) at (0,4) [nodal,fill=white] {};
\node (U3) at (0,5) [nodal,fill=white] {};
\node (U4) at (0,6) [nodal,fill=white] {};
\node (U5) at (0,7) [nodal,fill=white] {};
\node (U6) at (1,7) [nodal,fill=white] {};
\node (U7) at (0,8) [nodal,fill=white] {};
\node (U8) at (0,9) [nodal,fill=white] {};
\draw(U0)--(U1)--(U2)--(U3)--(U4)--(U5)--(U6) (U5)--(U7)--(U8);
\draw (S)--(U0);

\node (Y0) at (2,2) [nodal] {};
\node (Y1) at (2,3) [nodal] {};
\node (Y2) at (3,2) [nodal] {};
\node (Y3) at (4,2) [nodal] {};
\node (Y4) at (5,2) [nodal] {};
\node (Y5) at (6,2) [nodal, fill=red] {};
\node (Y6) at (7,2) {$\cdots$};
\node (Y7) at (2.0,4.26) {$\vdots$};
\draw(Y1)--(Y0)--(Y2)--(Y3)--(Y4)--(Y5);
\draw (S)--(Y0);

\end{tikzpicture}
\begin{tikzpicture}[scale=0.4]
\node (S) at (0,0) [nodal] {};

\node (T0) at (-2,2) [nodal,fill=white] {};
\node (T1) at (-2,3) [nodal,fill=white] {};
\node (T2) at (-2,4) [nodal,fill=white] {};
\node (T3) at (-2,5) [nodal,fill=white] {};
\node (T4) at (-2,6) [nodal,fill=white] {};
\node (T5) at (-2,7) [nodal,fill=white] {};
\node (T6) at (-3,7) [nodal,fill=white] {};
\node (T7) at (-2,8) [nodal,fill=white] {};
\node (T8) at (-2,9) [nodal,fill=white] {};
\draw(T0)--(T1)--(T2)--(T3)--(T4)--(T5)--(T6) (T5)--(T7)--(T8);
\draw (S)--(T0);

\node (U0) at (0,2) [nodal] {};
\node (U1) at (0,3) [nodal] {};
\node (U2) at (0,4) [nodal] {};
\node (U3) at (0,5) [nodal,fill=white] {};
\node (U4) at (0,6) [nodal,fill=white] {};
\node (U5) at (0,7) [nodal,fill=white] {};
\node (U6) at (1,7) [nodal,fill=white] {};
\node (U7) at (0,8) [nodal,fill=white] {};
\node (U8) at (0,9) [nodal,fill=white] {};
\draw(U0)--(U1)--(U2)--(U3)--(U4)--(U5)--(U6) (U5)--(U7)--(U8);
\draw (S)--(U0);

\node (X0) at (2,2) [nodal] {};
\node (X1) at (3,3) [nodal] {};
\node (X2) at (2,4) [nodal] {};
\node (X3) at (4,3) [nodal] {};
\node (X4) at (5,3) [nodal] {};
\node (X5) at (6,3) [nodal, fill=red] {};
\node (X6) at (7,3) {$\cdots$};
\draw(X0)--(X1)--(X2) (X1)--(X3)--(X4)--(X5);
\draw (S)--(X0);

\end{tikzpicture}
\begin{tikzpicture}[scale=0.4]
\node (S) at (0,0) [nodal] {};

\node (T0) at (-2,2) [nodal,fill=white] {};
\node (T1) at (-2,3) [nodal,fill=white] {};
\node (T2) at (-2,4) [nodal,fill=white] {};
\node (T3) at (-2,5) [nodal,fill=white] {};
\node (T4) at (-2,6) [nodal,fill=white] {};
\node (T5) at (-2,7) [nodal,fill=white] {};
\node (T6) at (-3,7) [nodal,fill=white] {};
\node (T7) at (-2,8) [nodal,fill=white] {};
\node (T8) at (-2,9) [nodal,fill=white] {};
\draw(T0)--(T1)--(T2)--(T3)--(T4)--(T5)--(T6) (T5)--(T7)--(T8);
\draw (S)--(T0);

\node (U0) at (0,2) [nodal,fill=white] {};
\node (U1) at (0,3) [nodal,fill=white] {};
\node (U2) at (0,4) [nodal,fill=white] {};
\node (U3) at (0,5) [nodal,fill=white] {};
\node (U4) at (0,6) [nodal,fill=white] {};
\node (U5) at (0,7) [nodal,fill=white] {};
\node (U6) at (1,7) [nodal,fill=white] {};
\node (U7) at (0,8) [nodal,fill=white] {};
\node (U8) at (0,9) [nodal,fill=white] {};
\draw(U0)--(U1)--(U2)--(U3)--(U4)--(U5)--(U6) (U5)--(U7)--(U8);
\draw (S)--(U0);

\node (D0) at (2,2) [nodal] {};
\node (D1) at (3,3) [nodal] {};
\node (D2) at (2,4) [nodal] {};
\node (D3) at (4,3) [nodal] {};
\node (D4) at (5,3) [nodal] {};
\node (D5) at (6,3) [nodal] {};
\node (D6) at (7,3) [nodal] {};
\node (D7) at (8,2) [nodal] {};
\node (D8) at (8,4) [nodal, fill=red] {};
\draw(D0)--(D1)--(D2) (D1)--(D3)--(D4)--(D5)--(D6)--(D7) (D6)--(D8);
\draw (S)--(D0);

\end{tikzpicture}
\end{center}

This leaves the cases of $ M_{\mathrm{root}} = E_{6}$ and $ M_{\mathrm{root}} = E_{7}$. In the first case, the vertices at distance $5$ from $s_{0}$ are intersected by a cusp given by a $\widetilde D_{7}$ diagram, while in the second case, the two vertices at distance $5$ and the vertex at distance $7$ are all intersected by a cusp given by an $\widetilde E_{8}$ diagram, as shown below:

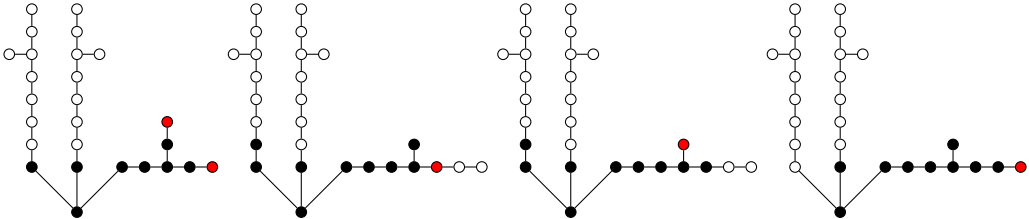
\begin{center}
\begin{tikzpicture}[scale=0.3]
\node (S) at (0,0) [nodal] {};

\node (T0) at (-2,2) [nodal] {};
\node (T1) at (-2,3) [nodal,fill=white] {};
\node (T2) at (-2,4) [nodal,fill=white] {};
\node (T3) at (-2,5) [nodal,fill=white] {};
\node (T4) at (-2,6) [nodal,fill=white] {};
\node (T5) at (-2,7) [nodal,fill=white] {};
\node (T6) at (-3,7) [nodal,fill=white] {};
\node (T7) at (-2,8) [nodal,fill=white] {};
\node (T8) at (-2,9) [nodal,fill=white] {};
\draw(T0)--(T1)--(T2)--(T3)--(T4)--(T5)--(T6) (T5)--(T7)--(T8);
\draw (S)--(T0);

\node (U0) at (0,2) [nodal] {};
\node (U1) at (0,3) [nodal,fill=white] {};
\node (U2) at (0,4) [nodal,fill=white] {};
\node (U3) at (0,5) [nodal,fill=white] {};
\node (U4) at (0,6) [nodal,fill=white] {};
\node (U5) at (0,7) [nodal,fill=white] {};
\node (U6) at (1,7) [nodal,fill=white] {};
\node (U7) at (0,8) [nodal,fill=white] {};
\node (U8) at (0,9) [nodal,fill=white] {};
\draw(U0)--(U1)--(U2)--(U3)--(U4)--(U5)--(U6) (U5)--(U7)--(U8);
\draw (S)--(U0);

\node (E0) at (2,2) [nodal] {};
\node (E1) at (3,2) [nodal] {};
\node (E2) at (4,2) [nodal] {};
\node (E3) at (4,3) [nodal] {};
\node (E4) at (4,4) [nodal, fill=red] {};
\node (E5) at (5,2) [nodal] {};
\node (E6) at (6,2) [nodal, fill=red] {};
\draw(E0)--(E1)--(E2)--(E3)--(E4) (E2)--(E5)--(E6);
\draw (S)--(E0);

\end{tikzpicture}
\begin{tikzpicture}[scale=0.3]
\node (S) at (0,0) [nodal] {};

\node (T0) at (-2,2) [nodal] {};
\node (T1) at (-2,3) [nodal] {};
\node (T2) at (-2,4) [nodal,fill=white] {};
\node (T3) at (-2,5) [nodal,fill=white] {};
\node (T4) at (-2,6) [nodal,fill=white] {};
\node (T5) at (-2,7) [nodal,fill=white] {};
\node (T6) at (-3,7) [nodal,fill=white] {};
\node (T7) at (-2,8) [nodal,fill=white] {};
\node (T8) at (-2,9) [nodal,fill=white] {};
\draw(T0)--(T1)--(T2)--(T3)--(T4)--(T5)--(T6) (T5)--(T7)--(T8);
\draw (S)--(T0);

\node (U0) at (0,2) [nodal] {};
\node (U1) at (0,3) [nodal,fill=white] {};
\node (U2) at (0,4) [nodal,fill=white] {};
\node (U3) at (0,5) [nodal,fill=white] {};
\node (U4) at (0,6) [nodal,fill=white] {};
\node (U5) at (0,7) [nodal,fill=white] {};
\node (U6) at (1,7) [nodal,fill=white] {};
\node (U7) at (0,8) [nodal,fill=white] {};
\node (U8) at (0,9) [nodal,fill=white] {};
\draw(U0)--(U1)--(U2)--(U3)--(U4)--(U5)--(U6) (U5)--(U7)--(U8);
\draw (S)--(U0);

\node (E0) at (2,2) [nodal] {};
\node (E1) at (3,2) [nodal] {};
\node (E2) at (4,2) [nodal] {};
\node (E3) at (5,2) [nodal] {};
\node (E4) at (5,3) [nodal] {};
\node (E5) at (6,2) [nodal, fill=red] {};
\node (E6) at (7,2) [nodal,fill=white] {};
\node (E7) at (8,2) [nodal,fill=white] {};
\draw(E0)--(E1)--(E2)--(E3)--(E4) (E3)--(E5)--(E6)--(E7);
\draw (S)--(E0);

\end{tikzpicture}
\begin{tikzpicture}[scale=0.3]
\node (S) at (0,0) [nodal] {};

\node (T0) at (-2,2) [nodal] {};
\node (T1) at (-2,3) [nodal] {};
\node (T2) at (-2,4) [nodal,fill=white] {};
\node (T3) at (-2,5) [nodal,fill=white] {};
\node (T4) at (-2,6) [nodal,fill=white] {};
\node (T5) at (-2,7) [nodal,fill=white] {};
\node (T6) at (-3,7) [nodal,fill=white] {};
\node (T7) at (-2,8) [nodal,fill=white] {};
\node (T8) at (-2,9) [nodal,fill=white] {};
\draw(T0)--(T1)--(T2)--(T3)--(T4)--(T5)--(T6) (T5)--(T7)--(T8);
\draw (S)--(T0);

\node (U0) at (0,2) [nodal] {};
\node (U1) at (0,3) [nodal,fill=white] {};
\node (U2) at (0,4) [nodal,fill=white] {};
\node (U3) at (0,5) [nodal,fill=white] {};
\node (U4) at (0,6) [nodal,fill=white] {};
\node (U5) at (0,7) [nodal,fill=white] {};
\node (U6) at (1,7) [nodal,fill=white] {};
\node (U7) at (0,8) [nodal,fill=white] {};
\node (U8) at (0,9) [nodal,fill=white] {};
\draw(U0)--(U1)--(U2)--(U3)--(U4)--(U5)--(U6) (U5)--(U7)--(U8);
\draw (S)--(U0);

\node (E0) at (2,2) [nodal] {};
\node (E1) at (3,2) [nodal] {};
\node (E2) at (4,2) [nodal] {};
\node (E3) at (5,2) [nodal] {};
\node (E4) at (5,3) [nodal, fill=red] {};
\node (E5) at (6,2) [nodal] {};
\node (E6) at (7,2) [nodal,fill=white] {};
\node (E7) at (8,2) [nodal,fill=white] {};
\draw(E0)--(E1)--(E2)--(E3)--(E4) (E3)--(E5)--(E6)--(E7);
\draw (S)--(E0);

\end{tikzpicture}
\begin{tikzpicture}[scale=0.3]
\node (S) at (0,0) [nodal] {};

\node (T0) at (-2,2) [nodal,fill=white] {};
\node (T1) at (-2,3) [nodal,fill=white] {};
\node (T2) at (-2,4) [nodal,fill=white] {};
\node (T3) at (-2,5) [nodal,fill=white] {};
\node (T4) at (-2,6) [nodal,fill=white] {};
\node (T5) at (-2,7) [nodal,fill=white] {};
\node (T6) at (-3,7) [nodal,fill=white] {};
\node (T7) at (-2,8) [nodal,fill=white] {};
\node (T8) at (-2,9) [nodal,fill=white] {};
\draw(T0)--(T1)--(T2)--(T3)--(T4)--(T5)--(T6) (T5)--(T7)--(T8);
\draw (S)--(T0);

\node (U0) at (0,2) [nodal] {};
\node (U1) at (0,3) [nodal,fill=white] {};
\node (U2) at (0,4) [nodal,fill=white] {};
\node (U3) at (0,5) [nodal,fill=white] {};
\node (U4) at (0,6) [nodal,fill=white] {};
\node (U5) at (0,7) [nodal,fill=white] {};
\node (U6) at (1,7) [nodal,fill=white] {};
\node (U7) at (0,8) [nodal,fill=white] {};
\node (U8) at (0,9) [nodal,fill=white] {};
\draw(U0)--(U1)--(U2)--(U3)--(U4)--(U5)--(U6) (U5)--(U7)--(U8);
\draw (S)--(U0);

\node (E0) at (2,2) [nodal] {};
\node (E1) at (3,2) [nodal] {};
\node (E2) at (4,2) [nodal] {};
\node (E3) at (5,2) [nodal] {};
\node (E4) at (5,3) [nodal] {};
\node (E5) at (6,2) [nodal] {};
\node (E6) at (7,2) [nodal] {};
\node (E7) at (8,2) [nodal, fill=red] {};
\draw(E0)--(E1)--(E2)--(E3)--(E4) (E3)--(E5)--(E6)--(E7);
\draw (S)--(E0);

\end{tikzpicture}
\end{center}

This shows that $E(L')=0$, and therefore the exceptional lattice of $L$ is trivial as well.

\end{proof}

\printbibliography

\end{document}